\def\timestamp{%
Time-stamp: <learning_CH-naw.tex: Monday 04-03-2019 at 10:30:56 (cet)>}
\def\stripname Time-stamp: <#1 #2>{#2}
\edef\filedate{\expandafter\stripname\timestamp}

\documentclass{naw-latex}

\usepackage{amsrefs}
\def\xxx{\relax\ifhmode\unskip\spacefactor3000 \space\fi} 
\newcommand\ZBL{\xxx ZBL\,}
\newcommand\JFM{\xxx JFM\,}
\newcommand\DOI{\\ DOI}

\DeclareMathSymbol\restr\mathbin{AMSa}{"16}
\DeclareMathSymbol\le   \mathrel{AMSa}{"36}    
\DeclareMathSymbol\ge   \mathrel{AMSa}{"3E}    

\newcommand\axiom{\mathsf}
\newcommand\CH{\axiom{CH}}

\newcommand\ZFC{\axiom{ZFC}}

\newcommand\calF{\mathcal{F}}

\DeclareMathSymbol\I0{AMSb}{`I}
\DeclareMathSymbol\N0{AMSb}{`N}

\newcommand\fin{\operatorname{fin}}
\newcommand\Opt{\operatorname{Opt}}
\newcommand\card[1]{\mathopen|{#1}\mathclose|} \let\abs=\card
\newcommand\bcard[1]{\bigl|{#1}\bigr|}
\newcommand\norm[1]{\mathopen\|{#1}\mathclose\|}

\newcommand\orpr[2]{\langle{#1},{#2}\rangle}
\newcommand\preim{^\gets}

\theoremstyle{plain}
\newtheorem{theorem}{Theorem}[section]
\newtheorem{proposition}[theorem]{Proposition}
\newtheorem{lemma}[theorem]{Lemma}

\theoremstyle{definition}
\newtheorem{definition}[theorem]{Definition}

\bibitem{machlearnCH}{
Shai Ben-David, Pavel Hrube\v{s}, Shay Moran, Amir Shpilka, 
and Amir Yehudayoff,
{\em Learnability can be undecidable},
Nature Machine Intelligence \textbf{1} (2019), 44--48,
   \DOI: 10.1038/s42256-018-0002-3.
}

\bibitem{Castelvecchi}{
Davide Castelvecchi,
{\em Machine learning leads mathematicians to unsolvable problem},
Nature \textbf{565} (2019), 277,
\DOI: 10.1038/d41586-019-00083-3.
}

\bibitem{MR1039321}{
Ryszard Engelking, 
{\em General topology, 2nd ed.},
Sigma Series in Pure Mathematics~6,
Heldermann Verlag, Berlin (1989),
\MR{1039321}.
}

\bibitem{MR0006493}{
Witold Hurewicz and Henry Wallman,
{\em Dimension Theory},
   Princeton Mathematical Series, v.~4,
   Princeton University Press, Princeton, N. J. (1941),
   \MR{0006493}.
}

\bibitem{MR597342}{
   Kenneth Kunen,
   {\em Set theory. An introduction to independence proofs},
   Studies in Logic and the Foundations of Mathematics~102,
   North-Holland Publishing Co., Amsterdam-New York (1980),
   \MR{597342}.
}

\bibitem{MR0048518}{
   Casimir Kuratowski,
   {\em Sur une caract\'{e}risation des alephs},
   Fundamenta Mathematicae~\textbf{38} (1951), 14--17,
   \MR{0048518},
   \DOI: 10.4064/fm-38-1-14-17.
}

\bibitem{MR0217751}{
   K. Kuratowski,,
   {\em Topology. Vol. I},
   Academic Press, New York-London; Pa\'{n}stwowe Wydawnictwo
   Naukowe, Warsaw (1966),
   \MR{0217751}.
}

\bibitem{KurUlam}{
 C. Kuratowski and St.~Ulam,
 {\em Quelques propri\'et\'es topologiques du produit combinatoire},
 Fundamenta Mathematicae~\textbf{19} (1932), 247--251,
 \ZBL{0005.18301}.
}

\bibitem{Reyzin}{
Lev Reyzin,
{\em Unprovability comes to machine learning},
Nature~\textbf{565} (2019), 166--167,
\DOI: 10.1038/d41586-019-00012-4.
}

\begin{document}

\StartArtikel[Titel={Machine learning and\\[4pt] the Continuum Hypothesis
                      \hfill{\font\a=cmssq8 \a(\filedate)}},
          AuteurA={K. P. Hart},
          AdresA={Faculteit EWI\crlf 
                  TU Delft\crlf
                  Postbus 5031\crlf
                  2600 GA  Delft},
          EmailA={k.p.hart@tudelft.nl},
          kolommen={2}
	  ]
	 
\StartLeadIn
In January 2019 the journal \textsl{Nature} reported on an exciting 
development in Machine Learning: the very first issue of journal
\textsl{Nature Machine Intelligence} contains a paper that describes
a learning problem whose solvability is neither provable nor refutable
on the basis of the standard $\ZFC$ axioms of Set Theory.

In this note I describe what the fuss is all about and indicate that maybe the 
problem is not so undecidable after all.
\StopLeadIn

\section*{Introduction}

In the paper \cite{machlearnCH}, in \textsl{Nature Machine Intelligence}, 
its authors exhibit an abstract machine-learning situation where the 
learnability is actually neither provable nor refutable on the basis of 
the axioms of~$\ZFC$.
This was deemed so exciting that the mother journal \textsl{Nature} 
actually devoted two commentaries to this: 
see~\cite{Reyzin} and~\cite{Castelvecchi}.

The first of these, \cite{Reyzin}, is rather matter-of-fact in its 
description of the problem but the second manages, in just a few lines, 
to mix up G\"odel's Incompleteness Theorems and the undecidability
of the Continuum Hypothesis.
It misstates the former --- ``G\"odel discovered logical paradoxes'' ---
and misinterprets the latter: ``a paradox known as the Continuum Hypothesis''.

In popular parlance one can say that G\"odel \emph{employed}
the Liar's Paradox in his proof of his incompleteness theorems but those
are not paradoxes, they are, as the name indicates, \emph{theorems}.
And the Continuum Hypothesis is not a paradox; it is `simply' a statement
that cannot be proved nor refuted on the basis of the usual axioms
of Set Theory.

In the first part of this note I will explain what the set theory behind 
the paper is and where the undecidability comes from.
In the second part I will show why I think that the problem is not 
undecidable at all: there is no algorithm that solves this particular
learning problem.

The amount of Set Theory needed to appreciate the arguments in this paper
is not too large.
We shall meet the cardinal numbers $\aleph_k$ for $k\le\omega$ as well
as the cardinality of the continuum: $2^{\aleph_0}$.
We shall also use the ordinals~$\omega_k$ as `typical' well-ordered sets
of cardinality~$\aleph_k$.
The first chapter of Kunen's book~\cite{MR597342} more than suffices
for our purposes.

\section{The Learning Problem}

The following is a summary of the parts of~\cite{machlearnCH} that lead
to the undecidability result.

The authors start with the following real-life situation as an instance of 
their general learning problem.
A website has a collection of advertisements that it can show to its visitors;
each advertisement, $A$, comes with a set, $F_A$, of visitors for whom it
is of interest: say if $A$ advertises running shoes then $F_A$~contains 
avid runners (or people who just like snazzy shoes).
Choosing the optimal advertisement to display amounts to choosing a finite
set from a population while maximizing the probability that the visitor
is actually in that set.
The problem is that the probabilty distribution is unknown.

Rather than dwell on this particular example the authors make an abstraction:
Given a set~$X$ and a family~$\calF$ of subsets of~$X$ find a member of~$\calF$
whose measure with respect to an unknown probability distribution
is close to maximal. 
This should be done based on a finite sample generated i.i.d.\ from the 
unknown distribution.

The undecidability manifests itself when we let $X$ be the unit interval~$\I$
and $\calF$ the family~$\fin\I$ of finite subsets of~$\I$.

\subsection{Learning functions}

In the general situation the abstract problem described above is made more
explicit and quantitative as follows.

For the unknown probability distribution~$P$ on~$X$ find $F\in\calF$
such that $E_P(F)$ is quite close to~$\Opt(P)$, which is defined
to be $\sup_{Y\in\calF}E_P(Y)$.

To quantify this further a \emph{learning function} for a~$\calF$
is defined to be a function
$$
G:\bigcup_{k\in\N}X^k\to\calF
$$
with certain desirable properties.

In this case the desirable properties are captured in the following definition
of an \emph{$(\epsilon,\delta)$-EMX learner} for~$\calF$.
This is a function~$G$ as above such that for some~$d\in\N$, 
depending on~$\epsilon$ and~$\delta$, the following
inequality holds
$$
\Pr\limits_{S\sim P^d}
\left[E_P\bigl(G(S)\bigr)\le\Opt(P)-\varepsilon\right]\le\delta
$$
for all distributions~$P$ with finite support.

The letters EMX abbreviate `estimating the maximum'.

\section{A combinatorial translation}

The first step in~\cite{machlearnCH} is to translate the existence of 
a suitable function~$G$ into a statement that is a bit more amenable
to set-theoretic treatment.

This translation involves what the author call monotone compression schemes.
Here and later we use $[X]^n$ to denote the family of $n$-element
subsets of~$X$.

\begin{definition}\label{def.scheme}
Let $m$ and $d$ be two natural numbers with $m>d$.
An \emph{$m\to d$ monotone compression scheme} for a family $\calF$ of finite
subsets of a set~$X$ is a function $\eta:[X]^d\to\calF$ such that whenever $A$~is
an $m$-element subset of~$X$ it has a $d$-element subset~$B$ such 
that $A\subseteq \eta(B)$, where we identify $B$ with a
point in~$X$ that enumerates it.
\end{definition}

This is slightly different from the formulation of Definition~2 
in~\cite{machlearnCH}, which leaves open the possibility that $\card{A}<m$ 
and that $\card{B}<d$, as it uses indexed sets.
It is clear from the results and their proofs that our definition
captures the essence of the notion.

There is a second unnamed function implicit in Definition~\ref{def.scheme}: 
the choice of the subset~$B$ of~$A$, we call this function~$\sigma$.
So our schemes consist of a pair of functions:
$\sigma:[X]^m\to[X]^d$ and $\eta:[X]^d\to\calF$; they should satisfy
$A\subseteq (\eta\circ\sigma)(A)$ for all~$A$.

Also, in the cases that we are interested in the set $X$ comes with a linear 
order~$\prec$; 
in that case we can identify $[X]^m$, the family of $m$-element subsets
with a subset of the product~$X^m$.
Every set corresponds to its monotone enumeration:
$[X]^m=\{x\in X^m:(i<j<m)\to(x_i\prec x_j)\}$.

The translation is now as follows.

\begin{lemma}[\cite{machlearnCH}*{Lemma~1.1}]
For an upward-directed family $\calF$ of finite sets the existence
of a $(\frac13,\frac13)$-EMX learning function is equivalent to the existence 
of a natural number~$m$ and an $(m+1)\to m$ monotone compression 
scheme for~$\calF$.  
\end{lemma}

The proof of necessity takes the natural number~$d$ in the learning function
and produces a monotone $(m+1)\to m$ compression scheme 
with $m=\lceil\frac32d\rceil$.

At this point the authors turn to the special case of the unit interval~$\I$
and its family $\fin\I$ of finite subsets and prove the following.

\begin{theorem}\label{thm.weakCH}
There is a monotone $(m+1)\to m$ compression scheme for~$\fin\I$ for some
$m\in\N$ if and only if $2^{\aleph_0}<\aleph_\omega$.
\end{theorem}

As the inequality $2^{\aleph_0}<\aleph_\omega$ is both consistent with
and independent of the axioms of~$\ZFC$ the same holds for the existence
of a compression scheme and for the existence of a $(\frac13,\frac13)$-EMX 
learning function.

Theorem~\ref{thm.weakCH} is an immediate consequence of the 
set of equivalences in the following theorem.

\begin{theorem}[\cite{machlearnCH}*{Theorem 1}]\label{Theorem-1}
Let $k\in\N$ and let $X$ be a set. 
Then there is a $(k+2)\to(k+1)$ monotone compression scheme for the finite
subsets of~$X$ if and only $\card{X}\le\aleph_k$.  
\end{theorem}

Indeed, $2^{\aleph_0}<\aleph_\omega$ if and only if $\card\I=\aleph_k$ for 
some $k\in\N$.

In the next section we take a closer look at monotone compression schemes
and point out a connection with an old result of Kuratowski's.

\section{On compression schemes and decompositions}

We begin by giving an equivalent description of monotone compression
schemes that does not mention the function~$\eta$.
This shows that it is~$\sigma$ that is doing the compressing.

\begin{proposition}
Let $m$ and $d$ be natural numbers and let $X$ be a set.
There is an $m\to d$ monotone compression scheme for the finite subsets
of~$X$ if and only if there is a finite-to-one function $\sigma:[X]^m\to[X]^d$
such that $\sigma(x)\subseteq x$ for all~$x$.  
\end{proposition}

\begin{proof}
If the pair $\orpr\eta\sigma$ determines an $m\to d$ monotone compression
scheme then $\sigma$ is finite-to-one.
For let $y\in[X]^d$ then $\sigma(x)=y$ implies $x\subseteq\eta(y)$, hence
there are at most $\binom{M}{m}$ such~$x$, where $M=\bcard{\eta(y)}$.

Conversely, if $\sigma$ is as in the statement of the proposition then
we can let $\eta(y)=\bigcup\{x:\sigma(x)=y\}$. 
\end{proof} 

\subsection{Kuratowski's decompositions}

The following theorem, proved by Kuratowski in~\cite{MR0048518}
provides one direction in his characterization of when a set has cardinality
at most~$\aleph_k$.

\begin{theorem}\label{thm.Kuratowski}
The power $\omega_k^{k+2}$ can be written as the union of $k+2$ sets,
$\{A_i:i<k+2\}$, such that for every~$i<k+2$ and every point 
$\langle x_j:j<k+2\rangle$ in $\omega_k^{k+2}$ the set of points~$y$ in~$A_i$
that satisfy $y_j=x_j$ for $j\neq i$ is finite. 
\end{theorem}

In Kuratowski's words ``$A_i$ is finite in the direction of the $i$th axis''.

\begin{proof}[Sketch of the proof]
The case $k=0$ is easy: 
let $A_0=\{\orpr mn:m\le n\}$ and $A_1=\{\orpr mn:m>n\}$.

The rest of the proof proceeds by induction on~$k$.
We give the step from $k=0$ to $k=1$ in some detail and leave the other
steps to the reader.

To decompose $\omega_1^3$ into three sets $A_0$, $A_1$ and $A_2$ we apply
the Axiom of Choice to choose (simultaneously) for each infinite 
ordinal~$\alpha$ in~$\omega_1$ a decomposition $\{X(\alpha,0),X(\alpha,1)\}$
of~$(\alpha+1)^2$, say by choosing well-orders of type~$\omega$ and then using
the decomposition obtained for~$k=0$.
\begin{itemize}
\item One puts $\langle\alpha,\beta,\gamma\rangle$ into $A_0$ if $\beta$ is the
largest coordinate and $\orpr\alpha\gamma\in X(\beta,0)$ or if $\gamma$~is 
the largest coordinate and $\orpr\alpha\beta\in X(\gamma,0)$.

\item One puts $\langle\alpha,\beta,\gamma\rangle$ into $A_1$ if $\alpha$ is the
largest coordinate and $\orpr\beta\gamma\in X(\alpha,0)$ or if $\gamma$~is 
the largest coordinate and $\orpr\alpha\beta\in X(\gamma,1)$.

\item One puts $\langle\alpha,\beta,\gamma\rangle$ into $A_2$ if $\alpha$ is the
largest coordinate and $\orpr\beta\gamma\in X(\alpha,1)$ or if $\beta$~is 
the largest coordinate and $\orpr\alpha\gamma\in X(\gamma,0)$.
\end{itemize}
To see that $A_0$ is finite in the direction of the $0$th coordinate take
$\orpr\beta\gamma\in\omega_1^2$, then 
$\langle\alpha,\beta,\gamma\rangle\in A_0$ implies
$\beta$~is largest and $\orpr\alpha\gamma\in X(\beta,0)$, or
$\gamma$~is largest and $\orpr\alpha\beta\in X(\gamma,0)$; in either case
$\alpha$~belongs to a finite set.

A similar argument works for~$A_1$ and $A_2$ of course.

The inductive steps for larger~$k$ are modelled on this step.
\end{proof}

We now show how Theorem~\ref{thm.Kuratowski} can be used to prove
sufficiency in Theorem~\ref{Theorem-1}.

\begin{proof}[Constructing a compression scheme from a decomposition]
From a decomposition as in Theorem~\ref{thm.Kuratowski} we construct
a finite-to-one function~$\sigma:[\omega_k]^{k+2}\to[\omega_k]^{k+1}$
such that $\sigma(x)\subseteq x$ for all~$x$.
We assume, without loss of generality, that the sets~$A_i$ are disjoint. 

Let $x\in[\omega_k]^{k+2}$ (so $i<j<k+2$ implies $x_i<x_j$).
Take (the unique) $i$ such that $x\in A_i$ and let $\sigma(x)$ be the point
in~$\omega_k^{k+1}$ that is~$x$ but without its coordinate~$x_i$.
In terms of sets we would have set $\sigma(x)=x\setminus\{x_i\}$.

This function is finite-to-one: if $y\in[\omega_k]^{k+1}$ then for each 
$i<k+2$ there are only finitely many~$x$ in~$A_i$ with $y=\sigma(x)$.
\end{proof}

As mentioned above Kuratowski's result works both ways: if $X^{k+2}$ admits 
a decomposition as above for~$\omega_k^{k+2}$ then $\card{X}\le\aleph_k$.
This suggests that the necessity in Theorem~\ref{Theorem-1} is related
to the converse of Theorem~\ref{thm.Kuratowski}.
This is indeed the case: one can construct a Kuratowski-type
decomposition from a compression scheme, but because of our definition
of the schemes we only get a decomposition of the subset $[\omega_k]^{k+2}$
of the whole power. 
This can be turned into one for the whole power but the process is a bit messy
so we leave it be.

The proof of necessity from~\cite{machlearnCH} closes the 
circle of implications that proves the following.

\begin{theorem}
For a set $X$ and a natural number $k$ the following are equivalent:
\begin{enumerate}
\item $\card{X}\le\aleph_k$,
\item $X^{k+2}$ admits a Kuratowski-type decomposition into $k+2$ sets,
\item there is a $(k+2)\to(k+1)$ monotone compression scheme
      for the finite subsets of~$X$.
\end{enumerate}
\end{theorem}

We sketch the proof of that last implication for completeness sake.
Both it and Kuratowski's necessity proof use a form of the following
lemma.

\begin{lemma}
Let $k$, $l$, and $m$ be natural numbers with $m>l$.
Assume $\sigma:[\omega_{k+1}]^{m+1}\to[\omega_{k+1}]^{l+1}$ determines 
an $(m+1)\to(l+1)$ monotone compression scheme.  
Then there is an $m\to l$ monotone compression scheme for~$\omega_k$.
\end{lemma}

\begin{proof}
We start by determining an ordinal~$\delta$ as follows.
Let $\delta_0=\omega_k$. 
Given $\delta_n$ use the fact that $\sigma$ is finite-to-one to find an 
ordinal $\delta_{n+1}>\delta_n$ such that every $x\in[\omega_{k+1}]^{m+1}$ that 
satisfies $\sigma(x)\in[\delta_n]^{l+1}$ is in $[\delta_{n+1}]^{m+1}$. 

In the end let $\delta=\sup_n\delta_n$.
Then $\delta$ satisfies: every $x\in[\omega_{k+1}]^{m+1}$ that 
satisfies $\sigma(x)\in[\delta]^{l+1}$ is in $[\delta]^{m+1}$.

\smallskip
We define an $m\to l$ monotone compression scheme for~$\delta$.
If $x\in[\delta]^m$ then $y=x\cup\{\delta\}$ is in $[\omega_{k+1}]^{m+1}$
and so $\sigma(y)\subseteq y$.
It is not possible that $\sigma(y)\subseteq x$ by the choice of~$\delta$,
hence $\delta\in\sigma(y)$ and so setting
$\varsigma(x)=\sigma(y)\setminus\{\delta\}$
defines a map $\varsigma:[\delta]^m\to[\delta]^l$.
This map is finite-to-one and satisfies $\varsigma(x)\subseteq x$ for all~$x$.
\end{proof}

To finish the proof of necessity we argue by induction and contradiction.
If $\card{X}=\aleph_{k+1}$ and there is a 
finite-to-one $\sigma:[X]^{k+2}\to[X]^{k+1}$ with $\sigma(x)\subseteq x$ 
for all~$x$
then there is a subset $Y$ of $X$ with $\card{Y}=\aleph_k$ and a
finite-to-one $\varsigma:[Y]^{k+1}\to[X]^k$ with $\varsigma(x)\subseteq x$ 
for all~$x$.
This would contradict the obvious inductive assumption.
We leave it as an exercise to the reader to ponder what absurdity would
arise in the case $k=0$.

\section{Algorithmic considerations}

In this section we address a point already raised by the authors 
in~\cite{machlearnCH}: the functions that are used in the previous sections
are quite arbitrary and not related to any recognizable algorithm.
Indeed, the constructions of the compression schemes for uncountable sets
blatantly applied the Axiom of Choice: once by assuming that the underlying 
sets were well-ordered and again when in every step of the induction a choice
of well-orders of type~$\omega_k$ needed to be made.

One may therefore wonder what happens if we impose some structure on the maps
in question.
One possible way of separating out `algorithmic' functions is by 
requiring them to have nice descriptive properties.
If `nice' is taken to mean `Borel measurable' then the desired
functions do not exist.

\subsection{Continuity and Borel measurability}

Here we show, for arbitray $m\in\N$, that there does not exist 
an $(m+1)\to m$ monotone compression scheme for the finite subsets of~$\I$ 
where the function~$\sigma$ is Borel measurable.
To this end let $m$ be a natural number and let $\sigma:[\I]^{m+1}\to[\I]^m$ 
be a function such that $\sigma(x)\subseteq x$ for all~$x$.

\begin{proof}[If $\sigma$ is continuous then $\sigma$ is not finite-to-one]
One can apply \cite{MR0006493}*{Theorem~VI.7} and deduce that there is a 
point~$y$ such that the fiber $\sigma\preim(y)$~is one-dimensional,
but in this case there is an elementary and more informative argument.

To this end let $x\in[\I]^{m+1}$ and assume for notational convenience 
that $\sigma(x)=\langle x_i:i<m\rangle$, i.e., that the coordinate~$x_m$ is 
left out of~$x$ when forming~$\sigma(x)$.

Let $\varepsilon=\frac13\min\{x_{i+1}-x_i:i<m\}$ and let $\delta>0$ be 
such that $\delta\le\varepsilon$ and 
for all $y\in[\I]^{m+1}$ with $\norm{y-x}<\delta$ we have 
$\norm{\sigma(y)-\sigma(x)}<\varepsilon$.

Now if $y\in[\I]^{m+1}$ and $\norm{y-x}<\delta$ then $\abs{y_i-x_i}<\varepsilon$
for all $i\le m$. Also, when $i<j$ we have $x_j-x_i>3\epsilon$.
It follows that $y_m-x_i>\epsilon$ for all $i<m$. 
This implies that $\sigma(y)=\langle y_i:i<m\rangle$ for all~$y$ with
$\norm{y-x}<\delta$.

This shows that for every $i$ the 
set $O_i=\bigl\{x\in[\I]^{m+1}:\sigma(x)=x\setminus\{x_i\}\bigr\}$
is open.
Because $[\I]^{m+1}$ is connected there is one~$i$ such that
$O_i=[\I]^{m+1}$.
This shows that $\sigma$ cannot be finite-to-one.
\end{proof}

The above proof can be used\slash adapted to show that if $\sigma$ is 
Borel measurable it is not finite-to-one either.

\begin{proof}%
[If $\sigma$ is Borel measurable then $\sigma$ is not finite-to-one]
There is a dense $G_\delta$-set $G$ in $[\I]^{m+1}$ such that the restriction
of~$\sigma$ to~$G$ is continuous, see~\cite{MR0217751}*{\S\,31~II}.

Let $x\in G$.
As in the previous proof we assume $\sigma(x)=\langle x_i:i<m\rangle$
and we obtain a $\delta>0$ such that $\sigma(y)=\langle y_i:i<m\rangle$
for all $y\in G$ that satisfy $\norm{y-x}<\delta$.

By the Kuratowski-Ulam theorem, \cite{KurUlam}, we can find a point~$y$
in~$G$ with $\norm{y-x}<\delta$ such that the set of points~$t$ in
the interval $(x_m-\delta,x_m+\delta)$ for which 
$y_t=\sigma(y)*\langle t\rangle$ belongs to~$G$ is co-meager.
But for every such point we have $\sigma(y_t)=\sigma(y)$ and
this shows that $\sigma$ is not finite-to-one.
\end{proof}

\subsection{EMX learning is impossible}

As we saw above a learning function is a function~$G$ from
the union~$\bigcup_{k\in\N}\I^k$ to the family of finite subsets
of~$\I$.
We can call such a function continuous or Borel measurable if its restriction
to each individual power is.

In the construction of an $(m+1)\to m$ compression scheme from a learning 
function the authors use its restriction to just one of these powers~$\I^d$,
where $d\le m$.
The definition of~$\eta(S)$ involves taking the union of $G(T)$ for all
$d$-element subsets~$T$ of~$S$, hence a union of $\binom md$ many sets.

The definition of~$\sigma$ involves choosing one $m$-element subset with
a certain property from of a given $m+1$-element set.

The latter choice can be made explicit using a Borel linear order on the
family of all finite subsets of~$\I$, or just $[\I]^m$.

An analysis of this procedure shows that if $G$ is Borel measurable
then so are $\sigma$ and $\eta$.

The results of this section then imply that a Borel measurable learning 
function does not exist.
In this author's opinion that means that the title of~\cite{machlearnCH}
should be emended to ``EMX-learning is impossible''.

\subsection{On the other hand \dots}

One may argue that the choice of the unit interval in~\cite{machlearnCH}
is a bit of a red herring.
None of the arguments in the paper use the structure of~$\I$ in any 
significant way.

In the step from the problem of the advertisements to the more abstract problem
there is no real need to go to the unit interval.
One may equally well use the set of rational numbers to code or rank
the elements of the learning set.

In that case there is, as we have seen, a $2\to1$ monotone compression scheme
for the finite subsets of~$\N$:
simply let $\sigma(x)=\max x$; the corresponding function~$\eta$
is defined by $\eta(n)=\{i:i\le n\}$.

It is an easy matter to transfer this scheme to the family of finite subsets
of the rational numbers.
Whether this scheme gives rise to a useful EMX learning function remains 
to be seen.

\completepublications

\end{document}